\documentclass[12pt]{amsart}
\usepackage{graphicx}
\usepackage{amssymb}
\usepackage{tikz}
\usetikzlibrary{matrix}
\usepackage{xcolor}
\usepackage{ytableau}

\usetikzlibrary{decorations.pathreplacing}

\definecolor{lightgray}{gray}{0.75}

\allowdisplaybreaks

\newcommand{\ds}{\displaystyle}

\newcommand{\Tr}{\mathrm{Tr}}
\newcommand{\bC}{\mathbb{C}}
\newcommand{\bZ}{\mathbb{Z}}
\newcommand{\tm}{\tilde{m}}
\newcommand{\bR}{\mathbb{R}}

\newcommand{\tM}{\widetilde{M}}

\newcommand{\im}{\mathrm{Im}}

\newcommand{\thebottomline}{\renewcommand{\thefootnote}{}
  \renewcommand{\footnoterule}{}
  \phantom{M}\footnotetext{\tiny{}\hfill
    \textit{\noindent\romannumeral\day.%
\romannumeral\month.\romannumeral\year}}}

\newtheorem{theorem}{Theorem}
\newtheorem{lemma}[theorem]{Lemma}
\newtheorem{corollary}[theorem]{Corollary}

\theoremstyle{definition}

\newtheorem{remark}[theorem]{Remark}

\title[Freely Independent Coin Tosses]
{Freely Independent Coin Tosses,\\[5pt] 
Standard Young Tableaux, and \\[5pt] the Kesten--McKay Law}

\author[i. s. arenas longoria]{Iris Stephanie Arenas Longoria}
\address{Department
  of Mathematics and Statistics, Queen's University, Jeffery
  Hall, Kingston, Ontario, K7L 3N6, Canada}

\email{18isal@queensu.ca}
\email{mingo@mast.queensu.ca} 
\thanks{Research supported by a Discovery Grant from the
  Natural Sciences and Engineering Research Council of
  Canada}

\author[j. a. mingo]{James A. Mingo} 
\address{Department
  of Mathematics and Statistics, Queen's University, Jeffery
  Hall, Kingston, Ontario, K7L 3N6, Canada}

\begin{document}

\begin{abstract} 
In this article, we shall start with a closed walk on a regular tree of degree $d$. These walks are described by the Kesten\-- McKay law which arises as the asymptotic distribution of a random $d$-regular graph on $n$ vertices. We will show that the moments of the Kesten \--McKay law are given by counting standard Young tableaux with at most 2 rows, and how some properties of the walk make sense even when $d$ is not an integer.  We will use free probability to instruct us how to build an explicit model in random matrix theory.
\end{abstract}

\maketitle

\section{Introduction.}

A \textit{walk} on a graph is a sequence of steps from one vertex to another  connected by an edge. The number of steps is the \textit{length} of the walk. The walk is \textit{closed} if it returns to its starting point. When we want to consider all walks with equal probability, we call these \textit{random walks}.  In analyzing walks, the first goal is to count the number of closed walks of a given length. 

If our graph is the infinite line, i.e., the vertices are indexed by the integers and  vertex $i$ is  connected to vertex $i+1$ by an edge, then the number of closed walks of length $2n$ is $\binom{2n}{n}$. If our graph is a half-line, then the number of closed walks of length $2n$ is the $n$th Catalan number $\frac{1}{n+1} \binom{2n}{n}$. If our graph has vertices indexed by the integer lattice $\bZ^2$, then the number of closed walks of length $2n$ is $\binom{2n}{n}^2$. 

\begin{figure}{t}
\includegraphics{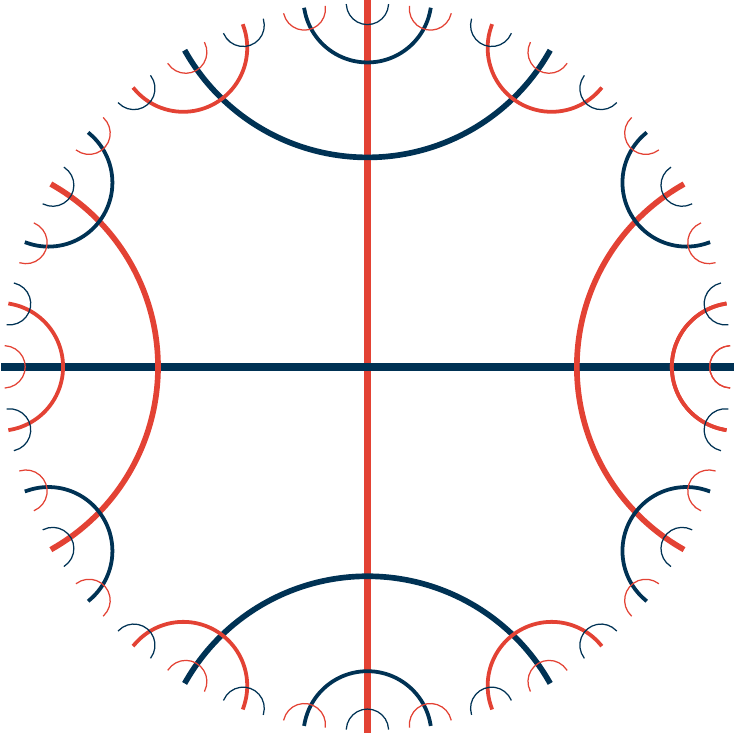}
\caption{\label{fig:4-tree}\small The 4-regular tree.}
\end{figure}

The graphs we shall be interested in are regular trees. See Figure \ref{fig:4-tree} for a drawing of a 4-regular tree. In the figure we only show the vertices 4 or fewer steps from the origin. Let $m_{2n}$ be the number of closed walks of length $2n$. We can see from the figure that $m_2 = 4$, $m_4 = 28$ and, with some more careful counting, that $m_6 = 232$. 

So far we have discussed the case of a $4$-regular tree, but we want to consider any $d$-regular tree. A crucial step is to let $c = d-1$ and to write $m_{2n}$ as a polynomial in $c$ with integer coefficients. This is where we see a connection to Pascal's triangle and from there to standard Young tableaux. Starting from the usual diagram, we remove all the entries to the right of the center line but retain the rule that each entry is the sum of the two entries in the row above. 

\setbox1=\hbox{{\tiny\begin{tikzpicture}
\matrix (m) [matrix of math nodes,row sep=1em,column sep=2em,minimum width=1em]
{
      &         &      &    &      &      &  1     \\
      &         &      &    &      & 1    &        \\
      &         &      &    & 1    &      &  1     \\
      &         &      & 1  &      & 2   &        \\
      &         &  1   &    & 3   &      &  2  \\
      &  1      &      & 4 &      & 5 &     \\
1     &         & 5   &    & 9 &      &  5  \\};
         \path (m-1-7) edge (m-2-6);
         \path (m-2-6) edge (m-3-7);
         \path (m-2-6) edge (m-3-5);
         \path (m-3-5) edge (m-4-4);
         \path (m-3-5) edge (m-4-6);
         \path (m-3-7) edge (m-4-6);         
         \path (m-4-4) edge (m-5-5);
         \path (m-4-4) edge (m-5-3);
         \path (m-4-6) edge (m-5-7);
         \path (m-4-6) edge (m-5-5);
         \path (m-5-3) edge (m-6-2);
         \path (m-5-3) edge (m-6-4);         
         \path (m-5-5) edge (m-6-4);
         \path (m-5-5) edge (m-6-6); 
         \path (m-5-7) edge (m-6-6);
         \path (m-6-2) edge (m-7-3);
         \path (m-6-2) edge (m-7-1);
         \path (m-6-4) edge (m-7-5);
         \path (m-6-4) edge (m-7-3);
         \path (m-6-6) edge (m-7-7);
         \path (m-6-6) edge (m-7-5);
\end{tikzpicture}}}
$\vcenter{\hsize=\wd1\box1}$ \hfill
\setbox2=\hbox{\begin{tabular}{l|l}
$n$ &  $m_{2n}$ \\ \hline
$0$ & $1$ \\
$1$ & $ 1+ c$ \\
$2$ & $ 1 + 3 c + 2 c^2$ \\
$3$ & $1 + 5 c + 9 c^2 + 5 c^3$ \\
\end{tabular}}
$\vcenter{\hsize=\wd2\box2}$\hfill\hbox{}

\medskip\noindent
If we look at the zeroth, second, fourth, and sixth rows we see that we have the same integers as in the table on the right. Our goal will be to prove that this holds for all rows and explain the connection between the two. 

We can deepen this relation by adding the parameter $c$ to the truncation of Pascal's triangle as in Figure \ref{fig:triangle_with_c}. In Section \ref{section:standard_young_tableau} we shall show that the coefficient of $c^k$ in the $n$th row is the number of standard Young tableaux with shape $(n-k,k)$; this is the main novelty of the paper.

The connection between the Catalan numbers and truncations of Pascal's triangle has been already studied (see \cite[Chapter 12]{k} and  \cite{t} for example); our contribution is to ``quantize'' this  diagram by the introduction of the continuous parameter $c$. For integer values of $c$ we can get the ordinary moment generating function for $\{m_{2n}\}_{n \geq 0}$, see Lemma \ref{lemma:m-generating-function}. Then we show that this is the same as a certain ordinary moment generating function for standard Young tableaux, see Theorem \ref{thm:main_syt}. This proves the pattern observed in Figure \ref{fig:triangle_with_c}. In particular, for non-integer values of $c$ for which there is no random walk.

In Section \ref{section:free_harmonic_analysis}, we examine the sequence $\{ m_{2n}\}_{n \geq 0}$ from the point of view of harmonic analysis. 
The moments are, for any $c > 0$, the moments of a probability measure on $\bR$, the \textit{Kesten-McKay law}. In Section \ref{section:random_matrix_theory}, we present the connection to random matrix theory and show how to model the law using the eigenvalue distribution of some random matrices. In Section \ref{section:free_probability}, we show how free probability gives a conceptual picture of why these models work and how to justify the claim that we have modelled freely independent coin tosses.

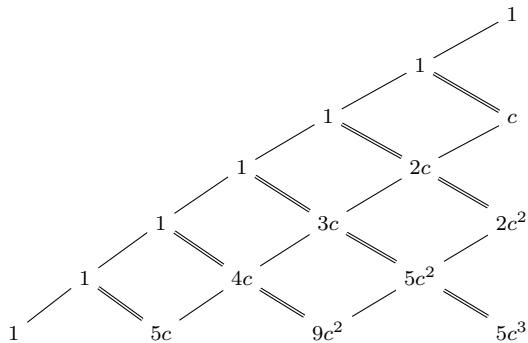
\begin{figure}
{\tiny\begin{tikzpicture}
\matrix (m) [matrix of math nodes,row sep=1em,column sep=2em,minimum width=1em]
{
      &         &      &    &      &      &  1     \\
      &         &      &    &      & 1    &        \\
      &         &      &    & 1    &      &  c     \\
      &         &      & 1  &      & 2c   &        \\
      &         &  1   &    & 3c   &      &  2c^2  \\
      &  1      &      & 4c &      & 5c^2 &     \\
1     &         & 5c   &    & 9c^2 &      &  5c^3  \\};
         \path (m-1-7) edge (m-2-6);
         \path (m-2-6) edge [double] (m-3-7);
         \path (m-2-6) edge (m-3-5);
         \path (m-3-5) edge (m-4-4);
         \path (m-3-5) edge [double] (m-4-6);
         \path (m-3-7) edge (m-4-6);         
         \path (m-4-4) edge [double] (m-5-5);
         \path (m-4-4) edge (m-5-3);
         \path (m-4-6) edge [double] (m-5-7);
         \path (m-4-6) edge (m-5-5);
         \path (m-5-3) edge (m-6-2);
         \path (m-5-3) edge [double] (m-6-4);         
         \path (m-5-5) edge (m-6-4);
         \path (m-5-5) edge [double] (m-6-6); 
         \path (m-5-7) edge (m-6-6);
         \path (m-6-2) edge [double](m-7-3);
         \path (m-6-2) edge (m-7-1);
         \path (m-6-4) edge [double](m-7-5);
         \path (m-6-4) edge (m-7-3);
         \path (m-6-6) edge [double](m-7-7);
         \path (m-6-6) edge (m-7-5);
\end{tikzpicture}}
\caption{\label{fig:triangle_with_c}%
We put double lines going down to the right to indicate that each element is $c$ times the element above and left plus the element to the above and right.}
\end{figure}

\section{An Ordinary Generating Function.}\label{sec:generating_function}

We let $\ds M(z) = 1 + \sum_{n=1}^\infty m_{2n} z^{2n}$ where $m_{2n}$ is the number of closed walks of length $2n$ on a $d$-regular tree. We will show that as a formal power series,
\[
M(z) = 
\frac{1}{2} \frac{(c- 1) - ( 1 + c)\sqrt{1 - 4 c z^2}}
{(1 + c)^2 z^2 - 1},\
\]
where $c = d - 1$. We will then show that the generating function for the even numbered rows of the triangle in Figure \ref{fig:triangle_with_c} satisfies the same equation.

Let $\tm_{2n}$ be the number of closed walks that return to the origin only once (at the end of the walk). Our idea is to write $m_{2n}$ as a sum of products of $\tm_{2k}$'s. For example, $m_2 = \tm_2$ because a closed walk of length 2 can only return once. Next $m_4 = \tm_4 + \tm_2 \tm_2$ because a walk of length 4 can return either once or twice. Likewise, $m_6 = \tm_2^3 + \tm_2 \tm_4 + \tm_4 \tm_2 + \tm_6$, where the first term counts the walks that return 3 times, the second and third terms count the walks that return 2 times, and the last term counts the walks that return only once. The point of doing this is that $\tm_{2n}$ is much easier to count. 

Consider a path in $\bR^2$ starting at $(0,0)$ and consisting of a sequence of  up-steps and down-steps where an  \textit{up-step} is  $(x, y) \mapsto (x + 1, y + 1)$ and a \textit{down step} is $(x, y) \mapsto (x+1 , y - 1)$. For $n \geq 0$, the number of paths from $(0, 0)$ to $(n, k)$ is $\binom{n}{(n+k)/2}$ if $n + k$ is even and $k \leq n$, and $0$ otherwise. Let us  recall the definition of a Dyck path. If $n, k \geq 0$, we say a path from $(0, 0)$ to $(n, k)$ is a \textit{weak Dyck path} if the path never goes below the line $ y = 0$.
A \textit{Dyck path} is a weak Dyck path that ends on the line $y=0$. An \textit{irreducible Dyck path} is a Dyck path that only returns once to the line $y=0$. It is a standard fact, see \cite[Corollary 6.2.3]{stan}, that the number of Dyck paths of length $2n$ is the $n$th Catalan number, $C_n$.

\begin{lemma}\label{lemma:1}
$\tm_{2n} = (1 + c) c^{n-1} C_{n-1}$, where $c = d -1$. 
\end{lemma}

\begin{proof}
Each closed walk of length $2n$ on a $d$-regular tree gives us a Dyck path of length $2n$. Indeed, each step away from the origin produces an up-step, each step closer to the origin produces a down-step. If the closed walk of length $2n$ only returns to the origin once, then we get an irreducible Dyck path.

We claim that the number of closed walks of length $2n$ that only return once to the origin and that get assigned to a given irreducible Dyck path is $(1 + c) c^{n-1} $. 

In fact, there are $d$ ways of choosing the first step; furthermore, there are $d-1$ ways of choosing each of the remaining $n-1$ steps away from the origin. The only way to move closer to the origin is to reverse a previous step, in the reverse order that the steps were taken. Since $c=d-1$, the number of walks assigned to a given irreducible Dyck path is $(1 + c) c^{n-1}$.

The number of irreducible Dyck paths of length $2n$ is counted by the $(n-1)$th Catalan number $C_{n-1}$, because by removing the first and last step from an irreducible Dyck path of length $2n$ we get a Dyck path of length $2n-2$; conversely, given a Dyck path of length $2n-2$, we create an irreducible Dyck path of length $2n$ by adding an up-step at the beginning and a down-step at the end. 

Thus the number of closed walks of length $2n$ on a $d$-regular tree is $(1+c)c^{n-1} C_{n-1}$. 
\end{proof}

Returning to our example where $d = 4$, we have $\tm_{2n} = 4 \cdot 3^{n-1} \cdot C_n$. Thus $\tm_2 = 4$, $\tm_4 = 12$, and $\tm_6 = 72$. Hence $m_2 = \tm_2= 4$, $m_4 = \tm_4 + \tm_2 \tm_2 = 28$, and $m_6 = \tm_2^3 + \tm_2 \tm_4 + \tm_4 \tm_2 + \tm_6 = 232$ as claimed above. This method is completely general.

Now denote by $N(z) = \sum_{n=0}^\infty C_n z^n$  the ordinary generating function of the Catalan numbers. The Catalan numbers satisfy the relation $C_n = \sum_{m=1}^n C_{n-1} C_{n-m}$, from which we obtain  that $N$ satisfies $z N(z)^2 + 1 = N(z)$ and hence 
$N(z) = \frac{1 - \sqrt{1 - 4z}}{2z}$, see \cite[Lecture 2]{ns}. Let $\tM(z) = \ds\sum_{n=1}^\infty \tm_{2n} z^{2n}$. Then

\begin{eqnarray*} 
\tM(z) & = & (1 + c) \sum_{n=1}^\infty C_{n-1} c^{n-1} z^{2n} 
= z^2 (1 + c) \sum_{n=1}^\infty C_{n-1} (c z^2)^{n-1}  \\
& = &
 z^2 (1 + c) \sum_{n=0}^\infty C_{n} (cz^2)^n = 
 z^2 (1 + c) N(cz^2) \\
 & = &
 \frac{1 + c}{2c}\Big( 1 - \sqrt{1 - 4 c z^2}\Big).\\
\end{eqnarray*}
And so
\begin{eqnarray*}\lefteqn{%
1 - \tM(z) =  1- \frac{1 + c}{2c}\Big( 1 - \sqrt{1 - 4 c z^2}\Big)} \\
& = &
\frac{1}{2c} \big(2c - (1 + c)( 1 - \sqrt{1 - 4 c z^2})\big) \\
& = &
\frac{c- 1 + ( 1 + c)\sqrt{1 - 4 c z^2}}{2c}
\end{eqnarray*}

We verify the following.

\begin{lemma}The number of closed walks of length $2n$ that return to origin exactly $k$ times is the coefficient on $z^{2n}$ in $\tM(z)^k$.
\end{lemma}

\begin{proof}
Let $W_{n,k}$ be the set of walks of length $2n$ that return to the origin exactly $k$ times. Thus, for $l_1,  l_2, \dots , l_k \geq 1$ such that $l_1 +  l_2 + \cdots + l_k =n$, let $W_{n,k}(l_1,  \dots , l_k)$ be the subset of $W_{n,k}$ consisting of walks that return to the origin at the $k$ points $\lbrace 2l_1 , 2l_2, 2l_3, \dots , 2l_k \rbrace$. By Lemma 1, $|W_{n,k}(l_1,  \dots , l_k)|=\tm_{2l_1} \tm_{2l_2} \cdots \tm_{2l_k}$. Then

\[
|W_{n,k}(l_1,  \dots , l_k)|z^{2n}=(\tm_{2l_1}z^{2l_1} ) \cdots (\tm_{2l_k} z^{2l_k});
\]

\begin{eqnarray*}
|W_{n,k}|z^{2n}&=&\sum_{\substack{l_1, \dots , l_k \geq 1 \\ l_1 + \cdots + l_k =n}}|W_{n,k}(l_1,  \dots , l_k)|z^{2n} \\
&=&
\sum_{\substack{l_1, \dots , l_k \geq 1 \\ l_1 + \cdots + l_k =n}}
(\tm_{2l_1}z^{2l_1} ) \cdots (\tm_{2l_k} z^{2l_k}).
\end{eqnarray*}
This last expression is exactly the term containing $z^{2n}$ in $(\tM(z))^k$.\end{proof}

Recall that $m_{2n}$ is the number of closed walks of length $2n$ on a $(1+c)$-regular tree, and $M(z)$ is the generating function of $\{m_{2n}\}$.

\begin{lemma}\label{lemma:m-generating-function}
\[
M(z) = 
\frac{1}{2} \frac{(c- 1) - ( 1 + c)\sqrt{1 - 4 c z^2}}
{(1 + c)^2 z^2 - 1}.
\]
\end{lemma}

\begin{proof}
First, we claim that $M(z)=1+ \sum_{k=1}^{\infty} \tM (z)^k $. Recall that $m_{2n}$ is the coefficient of $z^{2n}$ in $M(z)$. Let us recall some notation from the proof of Lemma 2. Let $W_{n,k}$ be the set of walks of length $2n$ that return exactly $k$ times. Then $m_{2n}=\sum_{k=1}^{\infty} |W_{n,k}|$. On the other hand, by the same lemma, the coefficient of $z^{2n}$ in $\sum_{k=1}^{\infty} \tM (z)^k$ is $\sum_{k=1}^{n} |W_{n,k}|$. Hence $m_{2n}$ is also the coefficient of $z^{2n}$ in $\sum_{k=1}^{\infty} \tM (z)^k$. Thus $M(z)=1+ \sum_{k=1}^{\infty} \tM (z)^k $. Thus

\begin{eqnarray*}
M(z) &=& 1 +  \sum_{n=1}^\infty \tM(z)^n = \frac{1}{1 - \tM(z)}\\
& = &
\frac{2c}{(c- 1) + ( 1 + c)\sqrt{1 - 4 c z^2}} 
=
\frac{1}{2} \frac{(c- 1) - ( 1 + c)\sqrt{1 - 4 c z^2}}
{(1 + c)^2 z^2 - 1}.
\end{eqnarray*}
\end{proof}

Now that we have the generating function for $M$ we may expand the formula for $M(z)$ in Lemma 3 as a Taylor series at $z=0$; we find that
\[
\begin{cases}
m_0= 1 \\
m_2=1+c \\
m_4= 1+3c+2c^2 \\
m_6 = 1 + 5 c + 9 c^2 + 5 c^3\\
\end{cases}.
\]
We note this agrees with the even-numbered rows in Figure \ref{fig:triangle_with_c}. 

\section{%
Standard Young Tableaux with at most two rows.
}
\label{section:standard_young_tableau}
Let us begin by recalling some basic facts about standard Young tableaux, see \cite[Chapter 1]{fult}. Let $n$ be a positive integer and $1 \leq \lambda_1 \leq \lambda_2 \leq \cdots \leq \lambda_k$ be such that $\lambda_1 + \lambda_2 + \cdots + \lambda_k = n$. Then $\lambda =(\lambda_1, \dots, \lambda_k)$ is a \textit{partition} of $n$ with $k$ \textit{parts}. To each partition of $n$ with $k$ parts we associate a Young diagram with $\lambda_1$ boxes in the first row, $\lambda_2$ boxes in the second row, and $\lambda_k$ boxes in the $k$th row. This is called the \textit{Young diagram with shape} $\lambda$. For example, when $\lambda = (3,1)$ the Young diagram is: 
\[\ydiagram{3,1}\,.\]
\begin{figure}[t]
\begin{center}
\includegraphics[scale=0.8]{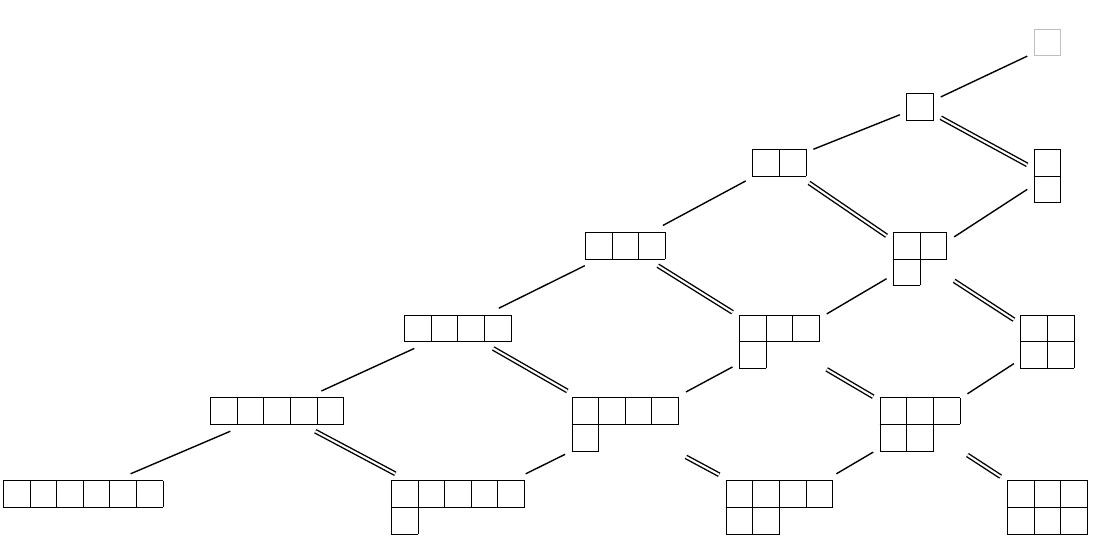} 
\end{center}
\caption{\label{fig:cell-diagram}\small Young diagrams with either one or 2 rows. A single line means a cell is added to the first row and a double line means a cell is added to the second row. The grey cell at the top represents the empty cell.}
\end{figure}
Given a partition $\lambda$ of $n$ we may place the integers $\{1, 2, 3, \dots, n\}$ into the cells of the Young tableau of shape $\lambda$ so that the numbers increase along rows (from left to right) and down columns. This is called a \textit{standard Young tableau} with shape $\lambda$. For example when $\lambda = (3,1)$ there are 3 ways of doing this:
\[  
\begin{ytableau} 1 & 2 & 4 \\ 3 \end{ytableau}\,, \quad
\begin{ytableau} 1 & 3 & 4 \\ 2 \end{ytableau}\,, \quad
\begin{ytableau} 1 & 2 & 3 \\ 4 \end{ytableau}\,.
\]
Given $\lambda$,  a partition of $n$, the number of standard Young tableaux of shape $\lambda$ is usually denoted $f^\lambda$, and there is an explicit formula for computing this, called the hook-length formula. See \cite[Corollary 7.21.6]{stan}. When $k = 1,2$, i.e., $\lambda$ has one or two parts, the formula can be given by a simple recursion which is at the heart of our result, see Figure \ref{fig:cell-diagram}. Suppose $\lambda = (\lambda_1, \lambda_2)$. The recursion stems from the observation that $n$ has to appear in a cell either at the end of the first row or at the end of the second row. If $\lambda_1 > \lambda_2$ and $n$ is in a cell at the end of the first row, then removing this cell gives us a standard Young tableau of shape $(\lambda_1 - 1, \lambda_2)$. If $n$ is in a cell at the end of the second row then removing this cell gives us a standard Young tableau of shape $(\lambda_1, \lambda_2 - 1)$. Hence $f^{(\lambda_1, \lambda_2)} = f^{(\lambda_1 -1, \lambda_2)} + f^{(\lambda_1, \lambda_2 -1)}$, provided we adopt the convention that $f^{(\lambda_1, 0)} = f^{(\lambda_1)}$ and $f^{(\lambda_1, \lambda_2)} = 0$ whenever $\lambda_1 < \lambda_2$. To start the induction, we let $f^{(0)} = 1$. If we let $g_{\lambda_1,\lambda_2} = \binom{\lambda_1 + \lambda_2}{\lambda_1} - \binom{\lambda_1 + \lambda_2}{\lambda_1 + 1}$, then by the Pascalian relation for binomial coefficients, we have $g_{\lambda_1,\lambda_2} = g_{\lambda_1 -1, \lambda_2} + g_{\lambda_1, \lambda_2 -1}$. Moreover $f$ and $g$ satisfy the same boundary conditions; hence $f = g$ or  
\[
f^{(\lambda_1,\lambda_2)} = \binom{\lambda_1 + \lambda_2}{\lambda_1} - \binom{\lambda_1 + \lambda_2}{\lambda_1 + 1}.
\]
From this one has $f^{(n)} = 1$ for all $n$ and $f^{(n,n)} = C_n$. See Figure \ref{fig:f-diagram} for a graphical representation of the recursion.

\begin{remark}\label{rem:dyck_connection}
 Note that by the refection principle \cite[Figure 1 on p. 69]{wf},  $f^{(\lambda_1, \lambda_2)}$  is also the number of weak Dyck paths from $(0,0)$ to $(\lambda_1 + \lambda_2, \lambda_1 - \lambda_2)$. There is also an explicit bijection where the numbers in the upper row of the standard Young tableau indicate the up-steps of the Dyck path and the numbers in the lower row of the tableau indicates the down-steps of the path. 
\end{remark}

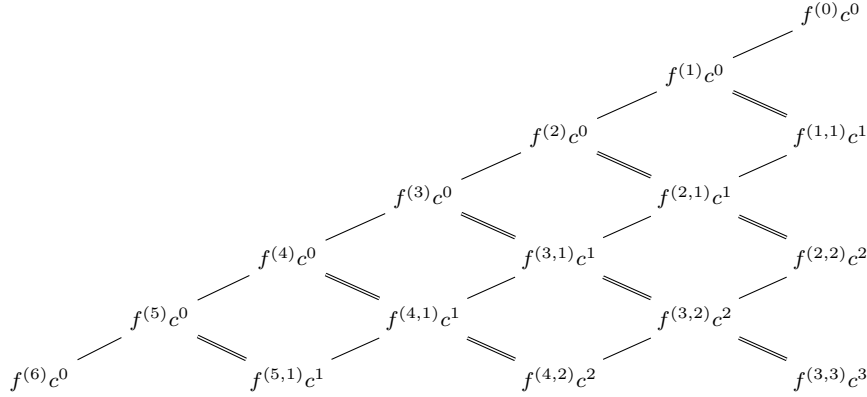
\begin{figure}[t]
\begin{center}
{\tiny\begin{tikzpicture}
\matrix (m) [matrix of math nodes,row sep=1em,column sep=2em,minimum width=1em]
{
      &         &      &    &      &      &  f^{(0)}c^0     \\
      &         &      &    &      &  f^{(1)}c^0   &        \\
      &         &      &    &  f^{(2)}c^0   &      &  f^{(1,1)}c^1     \\
      &         &      &  f^{(3)}c^0 &      &  f^{(2,1)}c^1   &        \\
      &         &  f^{(4)}c^0   &    &  f^{(3,1)}c^1   &      &  f^{(2,2)}c^2  \\
      &  f^{(5)}c^0      &      &  f^{(4,1)}c^1 &      &  f^{(3,2)}c^2   &     \\
f^{(6)}c^0     &         &  f^{(5,1)}c^1   &    &  f^{(4,2)}c^2   &      &  f^{(3,3)}c^3  \\};
         \path (m-7-1) edge (m-6-2);
         \path (m-6-2) edge (m-5-3);
         \path (m-5-3) edge (m-4-4);
         \path (m-4-4) edge (m-3-5);
         \path (m-3-5) edge (m-2-6);
         \path (m-2-6) edge (m-1-7);         
         \path (m-6-2) edge [double](m-7-3);
         \path (m-5-3) edge [double](m-6-4);
         \path (m-4-4) edge [double](m-5-5);
         \path (m-3-5) edge [double](m-4-6);
         \path (m-2-6) edge [double] (m-3-7);
         \path (m-6-4) edge [double](m-7-5);         
         \path (m-5-5) edge [double](m-6-6);
         \path (m-4-6) edge [double](m-5-7); 
         \path (m-7-3) edge (m-6-4);
         \path (m-6-4) edge (m-5-5);
         \path (m-5-5) edge (m-4-6);
         \path (m-4-6) edge (m-3-7);
         \path (m-7-5) edge (m-6-6);
         \path (m-6-6) edge (m-5-7);
         \path (m-6-6) edge [double](m-7-7);
\end{tikzpicture}}
\end{center}
\caption{\label{fig:f-diagram}\small The diagram illustrating the relationship between the $f^{(\lambda)}$'s: $f^{(\lambda_1,\lambda_2)}=  f^{(\lambda_1, \lambda_2-1)} + f^{(\lambda_1-1,\lambda_2)}$. This relation and the boundary conditions imply that $f^{(\lambda_1,\lambda_2)} = \binom{\lambda_1 + \lambda_2}{\lambda_1} - \binom{\lambda_1 + \lambda_2}{\lambda_1 + 1}$. See Remark \ref{rem:dyck_connection} for a connection to weak Dyck paths.}
\end{figure}

\begin{theorem}\label{thm:main_syt}
Let $m_{2n}$ be the number of closed walks of length $2n$ on a $(1+c)$-regular tree. Then
\[m_{2n} = \sum_{k=0}^nf^{(2n-k, k)} \, c^k
= \sum_{k=0}^n
\Big[\binom{2n}{k} -\binom{2n}{k-1}\Big] c^k.\]
In particular the coefficient of $c^k$ in $m_{2n}$ is the number of weak Dyck paths from $(0,0)$ to $(2n , 2n - 2k)$. 
\end{theorem}

\begin{proof}
Let $a_n$ be the sum of the terms in row $n$ of Figure \ref{fig:f-diagram}. So $a_0 = a_1 = 1$ and $a_2 = 1 + c$. In general $a_{2n} = \sum_{k=0}^n f^{(2n-k,k)}c^k$ and $a_{2n+1} = \sum_{k=0}^n f^{(2n+1-k,k)}c^k$. The relation $f^{(2n-k-1,k)} + f^{(2n-k,k-1)} = f^{(2n-k,k)}$ shows that $(1 + c) a_{2n-1} = a_{2n}$ and $(1 +c) a_{2n} = a_{2n+1} + f^{(n,n)}c^{n+1}$. Thus $a_{2n+1} = (1 +c) a_{2n} - C_n c^{n+1}$. Let $A(z) = \sum_{n=0}^\infty a_{2n} z^{2n}$ be the ordinary generating function of the even terms. Then
\begin{eqnarray*}
A(z) &=& 1 + \sum_{n=1}^\infty a_{2n} z^{2n} 
  = 
1 + \sum_{n=1}^\infty (1 + c) a_{2n -1} z^{2n} \\
& = & 
1 + \sum_{n=1}^\infty (1 + c)[(1 + c) a_{2(n-1)} - c^n C_{n-1}] z^{2n} \\
& = &
1 + (1 + c)^2 z^2 \sum_{n=0}^\infty a_{2n} z^{2n}  -
cz^2(1 + c) \sum_{n=0}^\infty C_{n} (c z^2)^{n} \\
& = &
1 + (1 + c)^2 z^2 A(z) - (1 + c)c z^2 N(c z^2)
\end{eqnarray*}
where $N(z) = \sum_{n=0}^\infty C_n z^n$ is the generating function of the Catalan numbers. Thus
\begin{eqnarray*}\lefteqn{
A(z)( z^2 (1 + c)^2 - 1) = (1 + c)cz^2 N(cz^2) -1 }\\
& = &
(1 + c)cz^2 \frac{ 1 - \sqrt{1 - 2 c z^2}}{2cz^2} - 1  
=
\frac{1}{2} \Big\{ c -1 - (1 + c) \sqrt{1 - 4 c z ^2} \Big\}.
\end{eqnarray*}
From Lemma \ref{lemma:m-generating-function} we have $A(z) = M(z)$. This shows that $m_{2n} = a_{2n}$.
\end{proof}

\begin{corollary}
Let $v_1$ and $v_2$ be two vertices in a $(1 + c)$-regular tree connected by one edge. Let $m_{2n-1}$ be the number of walks starting at $v_1$ and ending at $v_2$. Then $m_{2n-1} = \sum_{k=0}^{n-1} f^{(2n-1-k,k)} c^k$ and the coefficient of $c^k$ is the number of weak Dyck paths starting at $(0, 0)$ to $(2n -1, 2n -2k -1)$.
\end{corollary}

\begin{proof}
We only have to show $m_{2n} = (1 + c)m_{2n-1}$. By removing the last step from a closed walk of length $2n$ we get a walk of length $2n-1$ starting at $v_1$ and ending one vertex from $v_1$. There are $(1+ c)$ vertices where it might end. 
Moreover given a walk $w$ of length $2n -1$ starting at $v_1$ and ending one edge from $v_1$ there is a unique way to complete $w$ to a closed walk of length $2n$. Thus $m_{2n} = (1 + c) m_{2n -1}$. In the notation of Theorem \ref{thm:main_syt}: $a_{2n-1} = m_{2n-1}$. 
\footnote{Note to referee: the tricky point in implementing your suggestion would be to show that $(1 + c) m_{2n} = m_{2 n + 1} + C_n c^{n+1}$  but there is no physical model for $m_{2n+1}$.}
\end{proof}

\begin{remark}
There does not appear to be an explicit way to assign, to a closed walk of length $2n$ on a $d$-regular tree, a standard Young tableau. Theorem \ref{thm:main_syt} suggests that it is possible; our proof uses generating functions to prove Theorem \ref{thm:main_syt}. 
\end{remark}
So we shall leave this question open and relate these moments to some probability measures on the real line using classical tools now going under the name of free harmonic analysis, see \cite[Chapter 3]{vdn}. 

\section{Some Free Harmonic Analysis.}
\label{section:free_harmonic_analysis}
We have shown in Theorem \ref{thm:main_syt} that the number, $m_{2n}$, of closed walks on a $(1 + c)$-regular tree is given by the sum of the entries in the even-numbered rows of Figure \ref{fig:triangle_with_c}, the coefficients of which count standard Young tableaux of a given shape. Kesten \cite{KE} showed that these numbers are also the moments of a symmetric probability measure on $\bR$, now referred as the \textit{Kesten--McKay law}. We shall now show how to recover the law from the moments using what is known as Stieltjes inversion. This will show that for any $c \geq 0$ the moments of Theorem \ref{thm:main_syt} are the moments of a probability measure, although from the combinatorics we do not have any reason to expect this for non-integer values of $c$. 

Suppose that we have a probability measure on $\bR$; for example, let $\eta(t) = (2 \pi)^{-1} \sqrt{4 - t^2}$ for $|t| \leq 2$, and $0$ otherwise. This is the density of Wigner's famous semi-circle law. Then the odd moments of $\eta$  are $0$, and the
even moments are given by the Catalan numbers:
\[
m_{2n} = \int_{-2}^2 t^{2n} \eta(t)\, dt = C_n.
\]
Let $F$ be the ordinary moment generating function of these moments; then
\[
F(z) = \sum_{n=0}^\infty C_n z^{2n} = \frac{1 - \sqrt{1 - 4 z^2}}{2 z^2}.
\] To perform \textit{Stieltjes inversion} we let 
\begin{equation}\label{eq:cauchy_semi}
G(z) = \frac{1}{z} F\Big( \frac{1}{z} \Big) = \frac{z - \sqrt{z^2 - 4}}{2}.
\end{equation} 
and then observe that
\[
\eta(t) = \lim_{\epsilon \rightarrow 0^+} \frac{-1}{\pi} \im(G(t + i \epsilon)).
\]
$G$ is known as the \textit{Cauchy transform} of $\eta$ and leads our explorations into complex analysis.  We have skipped over a few technical points here, namely how to analytically extend the square root to the complex upper half-plane. This is one of those pleasant cases where the analysis  lets you do the algebra that your heart tells you to do; see \cite[\S 3.1]{ms}. 

We can now do the same for $M$,  the ordinary moment generating function of the Kesten--McKay law. The first step again is to compute the Cauchy transform
\begin{equation}\label{eq:cauchy}
G(z) = \frac{1}{z} M\Big( \frac{1}{z}\Big) = \frac{(1 -c)z + (1 + c)\sqrt{z^2 - 4c}}{2(z^2 - (1 + c)^2)}.
\end{equation}
Then we perform Stieltjes inversion as above and get a density:
\begin{eqnarray}\label{eq:rhoc1}
\rho_c(t) &=& \lim_{\epsilon \rightarrow 0^+} \frac{-1}{\pi} \im(G(t + i \epsilon)) \notag\\
&=&
\begin{cases}
\ds\frac{1 + c}{2 \pi} \frac{\sqrt{4 c - t^2}}{(1  + c)^2 - t^2} & \mbox{for } |t| \leq 2\sqrt{c} \\
0 & \mbox{for } |t| > 2\sqrt{c}
\end{cases}.
\end{eqnarray}
This is the density of the Kesten--McKay law for $c \geq 1$.

Now it is time to be careful about our parameter $c$. We assumed at the outset that we had a $d$-regular tree, and we let $c = d - 1 \geq 1$. When $d = 2$, we have the symmetric random walk on the line and  $\rho_1(t) = \frac{1}{\pi\sqrt{4 - t^2}}$ which is the density, supported on the interval $[-2, 2]$,  known as the \textit{arcsine law}. As mentioned in the introduction, the moment, $m_{2n}$, in this case counts the number of closed walks on the line and these, we know,  are given by the binomial coefficient $\binom{2n}{n}$. On the other hand, when $d = 2$, we have $c = 1$ and the sum in Theorem \ref{thm:main_syt} telescopes to $\binom{2n}{n}$, as $f^{(2n-k,k)} = \binom{2n}{k} -\binom{2n}{k-1}$. 

One can check by contour integration and residue calculus that for $c \geq 1$ we have $\int_{-2\sqrt c}^{2 \sqrt c}\rho_c(t) \, dt = 1$. Thus, for every real $c \geq 1$, $\rho_c$ is a probability density, and its moments are given by Theorem \ref{thm:main_syt} because we did not need to assume that $c$ was an integer in any of the proofs. 

Observe that for $c > 0$ we have $\rho_c(t) = \rho_{c^{-1}}(c^{-1} t)$ (see Figure \ref{fig:graph}). Thus, for $0 < c < 1$, we have
\[
\int_{-2\sqrt c}^{2 \sqrt c}  
\frac{1 + c}{2 \pi} \frac{\sqrt{4 c - t^2}}{(1  + c)^2 - t^2}
\, dt = c.
\] 
In this case we do \textit{not} get a probability measure. In fact, there is ``dark matter'' hiding somewhere, and the obvious place to look is at $\pm(1 + c)$, i.e., the poles of $G$ (see (\ref{eq:cauchy})). The dark matter turns out to be point masses, and we can compute the weights of these by observing that, for $G$ as in equation (\ref{eq:cauchy})
\[
\lim_{z \rightarrow 1 +c} (z- (1 + c))G(z) = (1 - c)/2.
\]

\begin{figure}
\includegraphics{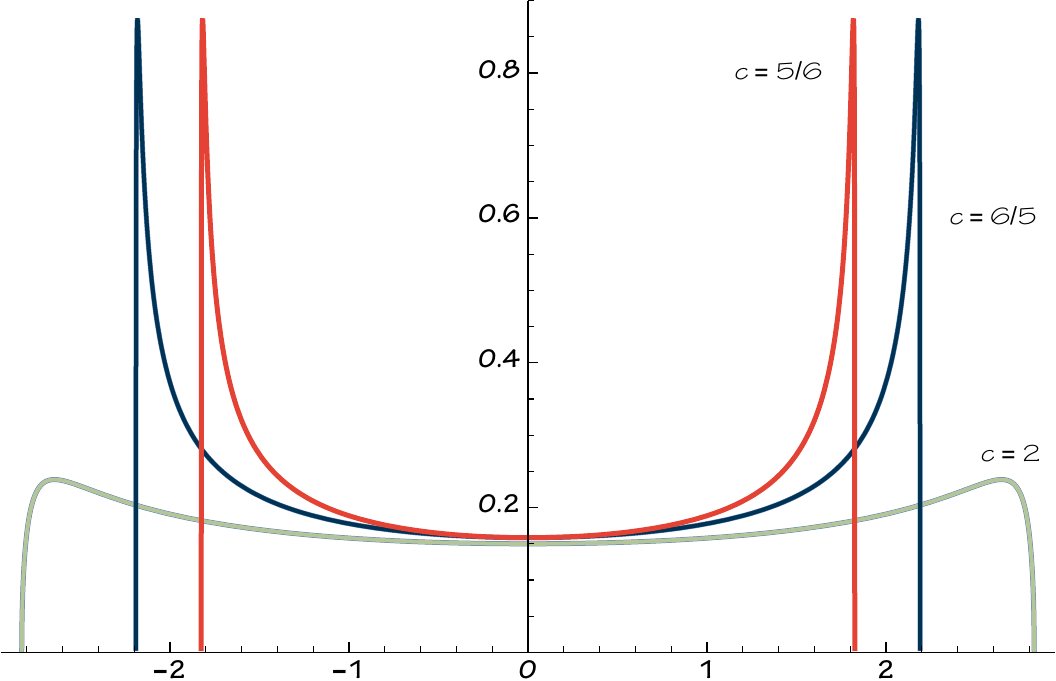}
\caption{\label{fig:graph}\small We have plotted the graph of $\rho_c$ for three values of $c$: $c = 2$ in green, $c = 6/5$ in blue, and $c = 5/6$ in red.}
\end{figure}
Indeed, that these limits give us the weights of the point masses is shown in \cite[Prop. 8]{ms}. Hence, for $0 < c < 1$, the Kesten--McKay law is given by
\begin{equation}\label{eq:rhoc2}
\frac{1 - c}{2} \delta_{-(1 + c)}
+
\frac{1 + c}{2 \pi} \frac{\sqrt{4 c - t^2}}{(1  + c)^2 - t^2}\, dt
+
\frac{1 - c}{2} \delta_{(1 + c)},
\end{equation}
where the continuous part is supported on the interval $[-2 \sqrt c, 2 \sqrt c]$. 

For any $ c> 0$, we now have a probability measure which we call \textit{the Kesten--McKay process} to reflect the fact that the parameter $c$ can be any positive real number. When $c = 0$, we get the distribution of a symmetric Bernoulli random variable with variance 1. See Figure \ref{fig:graph} for plots of $\rho_c$ for three different values of $c$. For non-integer values of $c$ there is no longer a random walk model, but we shall show that there \textit{is} a random matrix model.

\section{Explicit Realization: Random Matrices.}
\label{section:random_matrix_theory}

In this section, we consider the eigenvalue distribution of some $n \times n$ random matrices, and then observe what happens as $n \rightarrow \infty$. This gives us an explicit construction for each $c > 0$ of a random matrix model whose eigenvalue distribution converges to the corresponding Kesten--MacKay law. In Figure 5 we show some densities of the Kesten--Mackay law, and in Figures 6 and 7 we show histograms of the eigenvalue distribution when the matrix is $800 \times 800$. While we do not prove the convergence here the figures make the claim credible. 

A \textit{random matrix} is a matrix where the entries are random variables. An important example of such is a random rotation (or a random reflection), i.e., a random orthogonal matrix. An orthogonal matrix takes the unit sphere onto itself and we want to sample our orthogonal matrices, so that the probability that the north pole of our sphere is mapped into a given region is proportional to the Lebesgue measure of that region. This describes normalized Haar measure on the orthogonal group, and a random matrix produced in this way is said to be \textit{Haar distributed}, see \cite[Proposition 3.2.1]{kp}. If we consider the self-adjoint matrix $O + O^{-1}$, its eigenvalue distribution converges, as the size of the matrix tends to $\infty$,  to the arcsine law $\rho_1$. If we have independent Haar distributed random orthogonal matrices $O_1$ and $O_2$, then the eigenvalue distribution of $O_1 + O_1^{-1} + O_2 + O_{2}^{-1}$ converges to the Kesten--McKay law $\rho_3$. In general, if we have $l$ independent Haar distributed random orthogonal matrices $O_1, \dots, O_l$, the eigenvalue distribution of $O_1 + O_1^{-1} + \cdots + O_{l} + O_l^{-1}$ converges to $\rho_{2l -1}$. This gives us an explicit realization of $\rho_c$ when $c$ is an odd integer. In the next paragraph, we give a realization for any positive $c$. 

\begin{figure}
\includegraphics[scale=0.6]{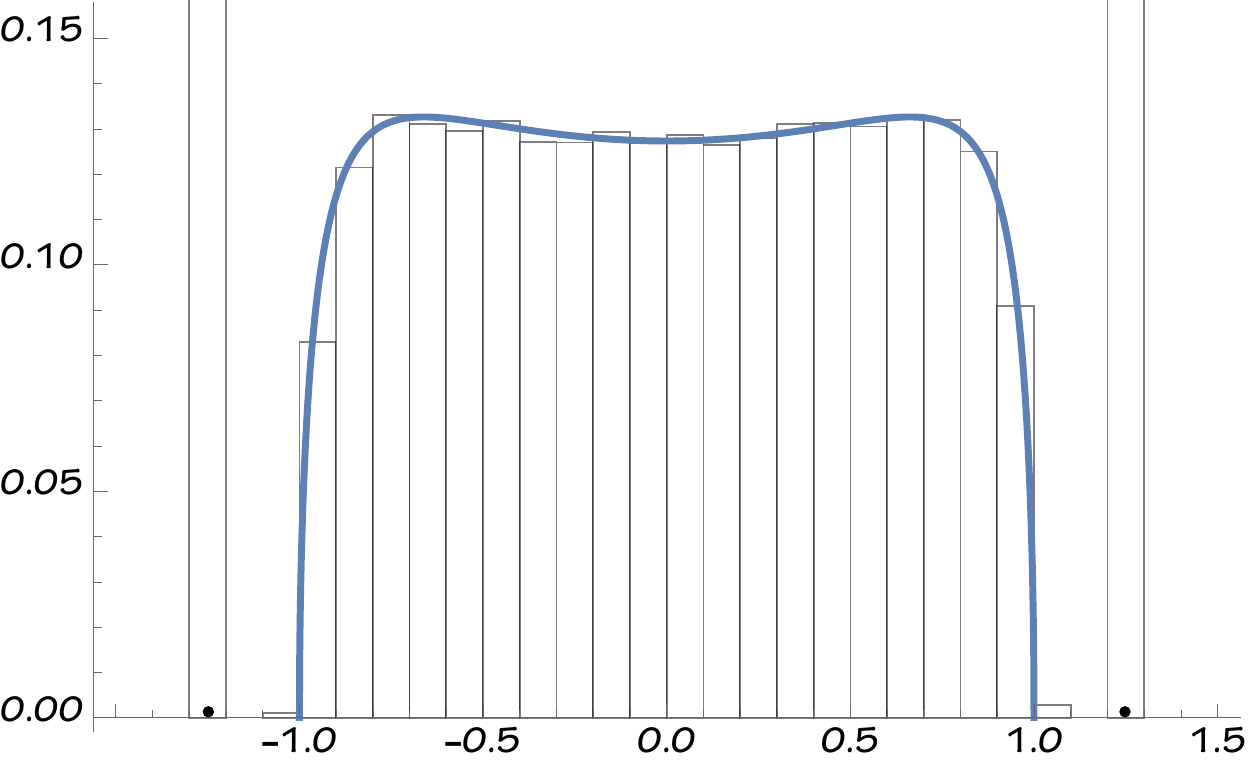}
\caption{\label{fig:kesten-1-4}\small{}The histogram shows the eigenvalue distribution of $X_{n,k}$ for $k = 800$ and $n = 1000$, i.e., $c = 1/4$. The curve is the graph of $\rho_{c}$. Note the two atoms at $\pm 5/4$; they each have mass $3/8$ and the center section has mass $c = 1/4$. To smooth the fluctuations we used 500 realizations.}
\end{figure}

Fix $c > 0$. Let $D$ be a $n \times n$ diagonal matrix with diagonal entries $\pm 1$ such that when $n$ is even the trace  $\Tr(D) = 0$, and when $n$ is odd $\Tr(D) = 1$. Let $O$ be an $n \times n$ Haar distributed random orthogonal matrix. Suppose $k = [n/(1 + c)]$ is the integer part of $n/(1 + c)$. Let $P$ be the projection diagonal matrix with the first $k$ diagonal entries being $1$ and the remaining $0$, so that $\Tr(P) = k$. Let $X_{n,k} = (n/k)\, PODO^{-1}P$, viewed as a $k \times k$ matrix in $M_k(\bC)$. The self-adjoint matrix $X_{n,k}$ will have $k$ eigenvalues, repeated according to multiplicity. In Figures \ref{fig:kesten-1-4} and \ref{fig:kesten-4-1} we have plotted a histogram showing the distribution of the (random) eigenvalues. The smooth curves show the density $\rho_c$ for either $c = 1/4$ (Figure \ref{fig:kesten-1-4}) or $c = 4$ (Figure \ref{fig:kesten-4-1}). That the curve fits the histogram so well is the conclusion of a general theorem on asymptotic free independence. In our last section, we shall show  these distributions can be described as the sum of freely independent Bernoulli random variables, of mean $0$ and variance $1$. 

\section{Free Probability.}
\label{section:free_probability}
Given a real valued random variable $X$, we can define its Cauchy transform $G$ using its ordinary moment generating function as was done in equation~(\ref{eq:cauchy}). This always works for bounded random variables; in the unbounded case there are more general techniques. We consider $G$ to be a function of a complex variable $z$, and when $X$ is bounded the domain of $G$ will be at least the complex plane less a closed interval of real numbers containing the range of $X$. Outside of some circle centered at the origin, $G$ will have a compositional inverse $K$, i.e., $K(G(z)) = G(K(z)) = z$ on appropriate domains (see \cite[\S 3.4]{ms}). At the origin, $K$ will always have a simple pole with residue $1$ and we let $R(z) = K(z) -1/z$.  $R$ is analytic on some open disk centered at $0$. This is Voiculescu's $R$-\textit{transform}.

When $X$ is the semi-circle law, $K(z) = z + z^{-1}$ and so $R(z) = z$; see equation (\ref{eq:cauchy_semi}). When $X$ is a symmetric Bernoulli random variable with variance $1$, we have that the ordinary moment generating function, the Cauchy transform, its inverse, and the $R$-transform can easily be found by solving a quadratic equation, and are given respectively by 
\[
M(z) = \ds\frac{1}{1 - z^2}, \quad
G(z) = \ds\frac{z}{z^2 - 1},
\]
\[
K(z) = \ds\frac{1 + \sqrt{1 + 4 z^2}}{2 z}, \quad
R(z) = \ds\frac{-1 + \sqrt{1 + 4 z^2}}{2 z}.
\]

\begin{figure}
\includegraphics[scale=0.8]{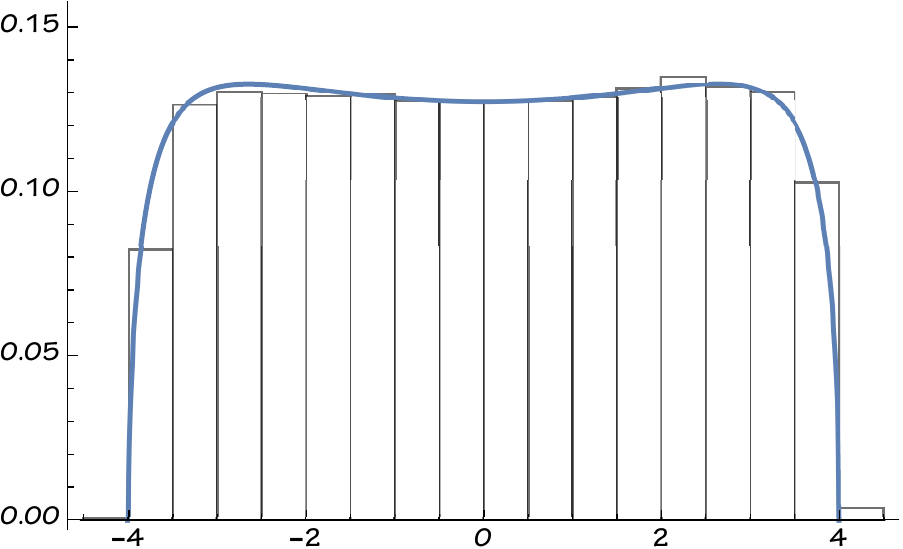}
\caption{\label{fig:kesten-4-1}\small The histogram shows the eigenvalue distribution of $X_{n,k}$ for $k = 200$ and $n = 1000$, i.e., $c = 4$. The curve is the graph of $\rho_{c}$. As with Fig.~\ref{fig:kesten-1-4} we have used 500 realizations.}
\end{figure}

When $X$ has the distribution of the Kesten--McKay law with parameter $c$ we have from equation~(\ref{eq:cauchy}), after solving another quadratic equation:
\[
K(z) = \frac{1 - c + (1 + c)\sqrt{ 1 + 4 z^2}}{2 z}.
\]
Thus for the Kesten--McKay law with parameter $c$ we have 
\begin{equation}\label{eq:r_transform_kesten}
R(z) = (1 + c)\ds \frac{-1 + \sqrt{1 + 4 z^2}}{2z},
\end{equation}
which is $1 + c$ times the $R$-transform of the Bernoulli random variable above. 

Voiculescu defined a new kind of independence for random variables, which he called \textit{free independence} because of its connection to free groups. (Note that Figure.~\ref{fig:4-tree} is the Cayley graph of the free group on 2 generators.) 

Let $\mathcal{A}$ be a unital algebra over $\mathbb{C}$ and $\phi: \mathcal{A} \rightarrow \mathbb{C}$ a linear functional such that $\phi(1_{\mathcal{A}})=1 \in \mathbb{C}$. We say that the pair $(\mathcal{A}, \phi)$ is a noncommutative probability space. Additionally, we will consider an antilinear operation $*$ such that $(a^*)^{*}=a$ ; $(ab)^{*}=b^{*}a^{*}$; and $\phi(a^* a)\geq 0$. Hence we say that $(\mathcal{A}, \phi)$ is a $*-$probability space. Set a fixed index $I$ and for each $i\in I$, let $A_i \subset \mathcal{A}$ be a unital subalgebra. The subalgebras $A_i$ are said to be freely independent if, for $k\in \mathbb{Z}^{+}$, 

\begin{equation*}
\label{def:free_indep}
\phi(a_1 \cdots a_k)=0
\end{equation*}

whenever $a_j \in A_{i_j}$ for some $i_j \in I$, $\phi(a_j)=0$ for all $1\leq j \leq k$, given that consecutive indices are different, i.e., $i_1 \leq i_2 \leq \cdots \leq i_k$.

The $R$-transform is so useful since it is linear relative to the addition of free random variables. Namely, if the random variables $X_1$ and $X_2$ are freely independent with $R$-transforms $R_1$ and $R_2$ respectively, then the $R$-transform of $X_1 + X_2$ is $R_1 + R_2$. In this way, the distribution of $X_1 + X_2$ may be computed by going back through $K$ and $G$ and then using Stieltjes inversion to find the distribution.

Voiculescu showed that many random matrix models exhibit this independence asymptotically, and thus free independence is now a key tool in random matrix theory. In particular, independent Haar distributed orthogonal matrices are asymptotically free, so $\rho_{2l-1}$ is the distribution of $l$ freely independent arcsine random variables as claimed in the previous section. 

In the case of continuous $c$, the $n \times n$ matrices $ODO^{-1}$ and $P$ (which is not random) are asymptotically freely independent, and thus their asymptotic distribution can be calculated by the rules of free probability; this is exactly the phenomenon observed in Figures \ref{fig:kesten-1-4} and \ref{fig:kesten-4-1}.  That the limit eigenvalue distribution of $\frac{n}{k} PODO^{-1}P$, regarded as a $k \times k$ matrix,  converges, as $n \rightarrow \infty$,  to the Kesten--McKay law follows from two results that we do not present here as it would be beyond the scope of this article. The first is that $ODO^{-1}$ is asymptotically free from $P$; this is shown in \cite[Theorem 36]{mp}. If we replace $O$ with a unitary matrix, a proof of asymptotic freeness can be found in \cite[Theorem 23.14]{ns}. That the $R$-transform of the limit is as claimed in equation~(\ref{eq:r_transform_kesten}) is proved in \cite[Example 11.35]{ns}. See also \cite[Example 12.8]{ns}. On the other hand, we can easily check, via the $R$-transform, that the limit distribution is the free additive convolution of symmetric Bernoulli random variables with variance $1$. If we interpret a single symmetric Bernoulli random variable with variance $1$ as a coin toss then the theorem below justifies the claim made in the introduction. 

\begin{theorem}
For an integer $c > 0$, the Kesten--McKay law is the free additive convolution of $d = 1 +c $ copies of the distribution of a symmetric Bernoulli random variable with variance $1$. 
\end{theorem}

\begin{proof}
We only have to check that the $R$-transform of $\rho_c$ is $(1 + c)$ times the $R$-transform of a symmetric Bernoulli random variable with variance $1$. This was already done in equation (\ref{eq:r_transform_kesten}). 
\end{proof}

\setbox1=\vtop{\hbox{\includegraphics{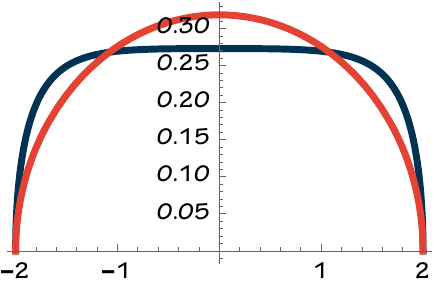}\ \ }}
\noindent
\hangindent\wd1\hangafter-6\hskip-\wd1\smash{%
\hbox to \wd1{\raise -6em\box1\hfil}}%
Sticking with this model, we can observe another central feature of free 
probability: the \textit{free central limit theorem} (see \cite[\S3.5] 
{vdn} or \cite[Lecture 8]{ns}). For $c > 1$,	 we let  $\mu_c(t) = \sqrt{c}    
\rho_c( \sqrt c t) = \eta(t) (1 - t^2/(\sqrt c + \sqrt{c^{-1}}))^{-1}$, where 
$\eta$ 
is the density of the semi-circle law given above. This yields a 
probability measure with support on $[-2, 2]$. The figure above shows both 
densities: the blue curve corresponds to $\mu_6$, and the red curve is the 
density of the semi-circle law. By letting $c \longrightarrow \infty$, we see 
that the Kesten--McKay density converges to that of the semi-circle law. This is 
a striking example of the free central limit theorem which says that the sum,
\[ \frac{X_1 + \cdots + X_c}{\sqrt c},\]
of $c$ freely independent random variables, suitably normalized, is asymptotically 
semi-circular as $c \rightarrow \infty$.

\thebottomline\end{document}